\documentclass[12pt]{amsart}
\pdfoutput=1
\usepackage[utf8]{inputenc}
\usepackage{etoolbox}

\makeatletter
 \def\@textbottom{\vskip \z@ \@plus 1pt}
 \let\@texttop\relax
\makeatother

\title{On almost Gallai colourings in complete graphs}

\author{Alexandr Grebennikov}

\address{ Institute of Science and Technology Austria (ISTA), Am Campus 1, 3400 Klosterneuburg, Austria.}
\email{aleksandr.grebennikov@ist.ac.at}

\author{Letícia Mattos}

\address{Institut für Informatik, Universität Heidelberg, Im Neuenheimer Feld 205, D-69120 Heidelberg, Germany}
\email{mattos@uni-heidelberg.de}

\author{Tibor Szabó}

\address{Institut für Mathematik, Freie Universität Berlin, Arnimallee 6, D-14195 Berlin, Germany}
\email{szabo@math.fu-berlin.de }

\usepackage{ mathrsfs }
\usepackage{ dsfont }
\usepackage{amsmath}
\usepackage{indentfirst}
\usepackage{mathtools}
\usepackage{amsfonts}
\usepackage{amsthm}
\usepackage{graphicx}
\usepackage{pgfplots}
\usepackage{mathrsfs}
\usepackage{lmodern}
\usetikzlibrary{arrows}
\usepackage[left=2.5cm,right=2.2cm,top=2.5cm,bottom=3.5cm]{geometry}
\usepackage[square,sort,comma,numbers]{natbib}
\usepackage{tikz}
\usepackage[breaklinks]{hyperref}
\hypersetup{
	colorlinks = true, 
	linkcolor = blue, 
	citecolor = blue 
}
\usepackage[english]{babel}
\usepackage{fullpage}
\usepackage{enumerate}
\usepackage{tasks}
\usepackage{color}
\usepackage{theoremref}
\usepackage{relsize}
\usepackage{setspace}

\usepackage{ amssymb }
\usepackage{ textcomp }
\usepackage[ruled]{algorithm2e} 	
\SetKwInput{KwInput}{Input}                
\SetKwInput{KwOutput}{Output}   
\usepackage{caption}

\newtheorem{theorem}{Theorem}

\newtheorem{lemma}[theorem]{Lemma}

\newtheorem{claim}[theorem]{Claim}

\numberwithin{theorem}{section}
\theoremstyle{definition}
\newtheorem*{remark}{Remark}

\newcommand*{\Scale}[2][4]{\scalebox{#1}{$#2$}}%

\def\bB{\mathbf{B}}

\def\K{\mathcal{K}}

\def\NN{\mathbb{N}}

\def\PP{\mathbb{P}}

\def\se{\subseteq}

\newcommand{\eps}{\varepsilon}

\newcommand{\Ex}[1]{\mathbb{E}\left(#1\right)}

\definecolor{lblue}{rgb}{0.5,0.5,1}

\newcommand{\eq}[1]{\begin{equation}\label{eq:#1}}
	\newcommand{\eqe}{\end{equation}}

\pgfplotsset{compat=1.18}

\begin{document}

\begin{abstract}
    We investigate a rainbow analogue of the Ruzsa--Szemerédi problem.
    For $t \in \mathbb{N}$, we say that a colouring of $E(K_n)$ is almost $t$-Gallai if no two rainbow $t$-cliques share an edge. 
    Motivated by a lemma of Berkowitz on bounding the modulus of the characteristic function of clique counts in random graphs, we study the maximum number $\tau_t(n)$ of rainbow $t$-cliques in an almost $t$-Gallai colouring of $E(K_n)$. 
    For every $t \ge 4$, we show that $n^{2-o(1)} \leq \tau_t(n) = o(n^2)$.
    For $t=3$, surprisingly, the behaviour is substantially different.
    Our main result establishes that
    $$\left ( \frac{1}{2}-o(1) \right ) n\log n \le \tau_3(n) \le O\big (n^{\sqrt{2}}\log n \big ),$$
    which gives the first non-trivial improvements over the simple lower and upper bounds.
    Our proof combines various applications of the probabilistic method and a generalisation of the edge-isoperimetric inequality for the hypercube.
\end{abstract}

\maketitle

	\section{Introduction}

    A colouring of the edges of a graph is called {\em Gallai colouring} if it admits no {\em rainbow triangle}, i.e. a triangle whose edges have pairwise distinct colours. 
    This term 
    was introduced by Gyárfás and Simonyi~\cite{gyarfas2004edge} due to the close connection of these colourings to the seminal work of Gallai~\cite{Gallai67} on comparability graphs, where he 
    obtained a structural classification of all colourings of the complete graph that avoid rainbow triangles. 
    Since its introduction, Gallai colourings have garnered significant attention and have been explored in various contexts, including Ramsey-type problems~\cite{chung1983edge, fujita2010rainbow, gyarfas2010gallai, gyarfas2010ramsey,cameron1986note}, extremal graph theory~\cite{FGP15,de2020number,balogh2019typical}, and graph entropy~\cite{KS00}.        
        
    We introduce a natural extension of Gallai colourings in which rainbow triangles, or more generally rainbow $t$-cliques, are allowed if they are edge-disjoint. In this setting, we are interested in maximising their number. We describe our original motivation for this concept in Theorem~\ref{th:characteristic_function}. However, we find the arising extremal problem attractive in its own right and make it the main focus of the current paper. 
    In the case of triangles (that is, when $t = 3$), our problem can be viewed as a rainbow analogue of the Ruzsa--Szemerédi problem, which asks for the maximum number of blue triangles in a red-blue colouring of $K_n$, under the condition that all blue triangles are edge-disjoint.

    For $t \in \mathbb{N}$, $t\geq 3$, and a graph $G$, we say that a function $c: E(G) \to \NN$ is an \emph{almost $t$-Gallai colouring} if no two rainbow $t$-cliques in it share an edge. Here a clique is called {\em rainbow} if its edges have pairwise distinct colours. 
    We denote by $\tau_t(n)$ the maximum number of rainbow $t$-cliques that an almost $t$-Gallai colouring can contain.
	Observe that $\tau_t(n) \ge \lfloor\frac{n}{t}\rfloor$, as one may     take $\lfloor\frac{n}{t}\rfloor$ vertex-disjoint $t$-cliques, each      coloured arbitrarily in a rainbow fashion, and colour all the           remaining edges with the same (arbitrary) colour. The trivial upper bound on $\tau_t(n)$ is $M(n,t)$: the maximum number of edge-disjoint $t$-cliques in $K_n$. Note that $M(n,t)$ is clearly at most $\binom{n}{2}/\binom{t}{2}$, and a classical result of Rödl~\cite{rodl1985} on the Erd\H{o}s--Hanani conjecture states that $M(n,t) = (1+o(1))\binom{n}{2}/\binom{t}{2}$ for fixed $t$ and $n \to \infty$.
      
    Our first theorem shows that for $t\geq 4$ neither of these simple bounds is tight. In particular, we show that $\tau_t(n)$ is subquadratic and determine its behaviour up to a factor of $n^{o(1)}$. 

	\begin{theorem}\label{th:at_least_4}
		We have
        \begin{enumerate}
                \item[$(i)$] $\tau_t(n) = o(n^{2})$ for every fixed $t \ge 3$, and
                \item[$(ii)$] $\tau_t(n) \ge n^{2-o(1)}$ for every fixed $t \ge 4$.
        \end{enumerate}
		Moreover, the construction for the lower bound uses only $\binom{t}{2}$ colours.
	\end{theorem}

    For the upper bound, we use the Graph Removal Lemma, proved independently by Alon, Duke, Lefmann, Rödl and Yuster~\cite{alon1994algorithmic}, and Füredi~\cite{furedi1995extremal}.
    For the lower bound, we employ the construction of Kov\'acs and Nagy~\cite{KN22} based on the $(3, h)$-gadget-free sets introduced by Alon and Shapira \cite{AlonShapira05}, which in turn relies on a variant of the large $3$-AP-free sets of Behrend.

    Interestingly, for $t=3$ the lower bound construction fails to work. In our main theorem we show that this is not a coincidence and improve the upper bound by a polynomial factor. We also improve the trivial linear lower bound by a logarithmic factor. 
    All logarithms in this paper are in base 2.

	\begin{theorem}
		\label{th:rainbow_triangles}
			We have 
			\[
				\left (\dfrac{1}{2} - o(1) \right ) \cdot n \log n \le \tau_3(n) \le O(n^{\sqrt{2}} \log n).
			\]
			Moreover, the construction for the lower bound uses only three colours.
		\end{theorem}

    For our lower bound we construct an almost 3-Gallai colouring of $K_n$ containing $(1/2 - o(1)) n \log n$ rainbow triangles by embedding one of the colours into the hypercube graph of dimension $\log n$.    
    For the upper bound we show that in any almost 3-Gallai colouring of $K_n$ there are $O(n^{\sqrt{2}} \log n)$ rainbow triangles. The proof combines various applications of the probabilistic method and a generalisation of the edge-isoperimetric inequality for the hypercube due to Bernstein~\cite{bernstein1967maximally}, Harper~\cite{harper1964optimal}, Hart~\cite{hart} and Lindsey~\cite{lindsey1964assignment}.
    
    We believe that our lower bound lies closer to the true value of $\tau_3(n)$, and therefore we conjecture that the stronger upper bound of $O(n\log n)$ should hold.

    \begin{remark}
    It is worth noting that when the host graph is not complete, the maximum number of rainbow triangles in an almost $3$-Gallai colouring can differ significantly from the bounds in Theorem~\ref{th:rainbow_triangles}.
    For instance, consider the tripartite graph $H$ of Ruzsa and Szemerédi~\cite{ruzsa1978triple}, 
    on $n$ vertices with $n^{2-o(1)}$ edges such that 
    each edge appears in exactly one triangle.
    Colouring the edges of every triangle of $H$ in three different colours, we obtain an almost $3$-Gallai colouring of $H$. Indeed, all the triangles in $H$ are rainbow and also pairwise edge-disjoint. 
    \end{remark}

		 Finally, we demonstrate an application of Theorems~\ref{th:at_least_4} and~\ref{th:rainbow_triangles}, which in fact was our original motivation for the concept of almost Gallai colourings. We obtain an upper bound on the modulus of the characteristic function of clique counts in the binomial random graph $G(n,p)$, for constant $p$. 
		 This improves a result of Berkowitz~\cite[Lemma 18]{berkowitz2018local}.
		
		 \begin{theorem}\label{th:characteristic_function}
			Let $p \in (0,1)$ and $t \ge 4$ be constants, and let $X_t$ be the number of $t$-cliques in $G(n,p)$.
			Then, for $s \in [-\pi, \pi]$ we have
			\begin{align*}
				\big | \Ex{e^{isX_3}} \big | \le \exp \left ( - \Omega \big (s^2 n \log n \big ) \right )  \qquad \text{and} \qquad \big | \Ex{e^{isX_t}} \big | \le \exp \left ( - \Omega \big ( s^2 n^{2-o(1)} \big ) \right ),
			\end{align*}
			where the implicit constants depend only on $p$ and $t$.
		 \end{theorem}
        This theorem implies that $\big|\Ex{e^{isX_t}}\big|$ is exponentially small for $s \ge n^{-1+\eps}$ and $t \ge 4$. Such estimates (combined with estimates for other ranges of $s$ obtained by different methods, see \cite{berkowitz2018local}) are commonly used to prove anticoncentration results for the random variable $X_t$. 
        The idea of the proof is to combine our constructions for the lower bounds on $\tau_t(n)$ given by Theorems~\ref{th:at_least_4} and~\ref{th:rainbow_triangles} and Berkowitz's decoupling trick~\cite{berkowitz2018local}.
	
		The rest of the paper is organised as follows.
        In Section~\ref{sec:t_ge_4}, we prove Theorem~\ref{th:at_least_4}.
		In Section~\ref{sec:construction}, we provide the construction for the lower bound in Theorem~\ref{th:rainbow_triangles}.
		In Section \ref{sec:outline_upper}, we prove Theorem~\ref{th:rainbow_triangles} assuming some technical theorem and lemmas, which are then proved in Sections~\ref{sec:preliminary} and~\ref{sec:lemma-balanced}. 
		  Finally, in Section~\ref{sec:characteristic_function} we prove Theorem~\ref{th:characteristic_function}.

	\section{Proof of Theorem \ref{th:at_least_4}}
	\label{sec:t_ge_4}

        First we prove the upper bound.
        Fix an almost $t$-Gallai colouring $c: E(K_n) \to \mathbb{N}$.
        The idea is to construct a subgraph $G$ of $K_n$ in which (1) the number of rainbow $t$-cliques in this colouring is of the same order as in the original graph $K_n$, and (2) every $t$-clique is rainbow. Since these cliques are pairwise edge-disjoint, we can then apply the Clique Removal Lemma to $G$.        
        
        Let $V_1\cup \cdots \cup V_t = V(K_n)$ be a random partition of the vertex set into $t$ parts, where each vertex $v \in V(K_n)$ is independently assigned to $V_i$ with probability $1/t$ for all $i = 1,\ldots,t$.
        Similarly, let $\bigcup_{1 \le i < j \le t} C_{ij} = c(E(K_n))$ be a random partition of the colours used by $c$ into $\binom{t}{2}$ parts, where each colour is independently assigned to $C_{ij}$ with probability $\binom{t}{2}^{-1}$.

        Let $G$ be the random $t$-partite subgraph of $K_n$ with parts $V_1,\ldots, V_t$, where an edge between $V_i$ and $V_j$ is kept only if its colour was assigned to the part $C_{ij}$. Formally, 
        \begin{align*}
            E(G):=\bigcup_{1\leq i < j\leq t} \{e \in E(K_n[V_i,V_j]) : c(e) \in C_{ij}\}.
        \end{align*}
        Note that every $t$-clique in $G$ is a rainbow clique in $K_n$. 
        Indeed, every $t$-clique in $G$ has exactly one vertex in each $V_i$, and hence any two of its edges have different colours since
		$C_{ij} \cap C_{k\ell} = \emptyset$ for $\{i, j\} \neq \{k, \ell\}$.

    Next we show that, in expectation, a constant proportion of the rainbow $t$-cliques in $K_n$ remains in $G$. Fix an arbitrary rainbow $t$-clique $K$ induced by $c$ in $K_n$, and denote its vertices by $v_1,\ldots,v_t$. 
    Let $S_t$ be the set of all permutations on $\{1,\ldots,t\}$.
    The probability that $K \se G$ is precisely
    \begin{align*}
        \mathbb{P}(K \se G) & = \sum \limits_{\sigma \in S_t}  \mathbb{P} \Big ( v_i \in V_{\sigma(i)} \text{ and } c(v_iv_j) \in C_{\sigma(i)\sigma(j)} \text{ for all } 1 \le i < j \le t \Big ) = \dfrac{t!}{t^t} \cdot \dfrac{1}{\binom{t}{2}^{\binom{t}{2}}}.
    \end{align*}
        Above we used that $|S_t|=t!$, that the probability of each vertex $v_i$ being assigned to the part $V_{\sigma(i)}$ for every $i \in \{1,\ldots, t\}$ is $t^{-t}$, that the probability of each edge colour $c(v_iv_j)$ belonging to the correct colour class $C_{\sigma(i)\sigma(j)}$ for every $ 1 \le i < j \le t$ is $\binom{t}{2}^{-\binom{t}{2}}$, and that these events are mutually independent.

        Let $\rho_t(c)$ denote the number of rainbow $t$-cliques in $K_n$ induced by $c$.
        By linearity of expectation, the expected number of $t$-cliques in $G$ is
        \begin{align}
            \dfrac{t!}{t^t} \cdot \dfrac{1}{\binom{t}{2}^{\binom{t}{2}}} \cdot \rho_t(c). 
        \end{align}
        Therefore, there exists a $t$-partite subgraph $G^* \se K_n$ containing at least $\Omega(\rho_t(c))$ rainbow $t$-cliques. Moreover, all these cliques are edge-disjoint due to the almost $t$-Gallai property of $c$. 
        
        Recall that, by construction, all $t$-cliques in $G^*$ are rainbow, hence we conclude that all $t$-cliques of $G^*$ are pairwise edge-disjoint.
        Consequently, the total number of $t$-cliques in $G^*$ is at most $\binom{n}{2} = o(n^t)$.
        The Clique Removal Lemma ~\cite{alon1994algorithmic,furedi1995extremal} states that if a graph has $o(n^t)$ $t$-cliques, then it is possible to remove $o(n^2)$ of its edges to make it $K_t$-free. 
        However, we need to remove at least one edge from each of the $\Omega(\rho_t(c))$ pairwise edge-disjoint $t$-cliques in $G^*$ to make it $K_t$-free, which yields $\rho_t(c) = o(n^2)$.

        \medskip
        
        For the lower bound we consider a construction due to Kov\'acs and Nagy \cite[Theorem 1.3]{KN22} of a graph $G$ on $n$ vertices with the following properties:
	\begin{enumerate}
		\item [$(1)$] $G$ contains $n^2/e^{O(\sqrt{\log n})}$ copies of $t$-cliques;
		\item [$(2)$] each edge of $G$ belongs to exactly one $t$-clique;
		\item [$(3)$] any triangle of $G$ fully lies inside one of the $t$-cliques.
	\end{enumerate}
	
	Now we construct an almost $t$-Gallai colouring of $K_n$ using the graph $G$ as follows. 
	Consider an arbitrary embedding of $G$ into $K_n$, and colour the edges on each $t$-clique of $G$ with pairwise distinct colours $1, 2, \ldots, \binom{t}{2}$. 
	Colour all the edges not in $E(G)$ with colour $1$.
	By property (1), this colouring contains the desired number of rainbow copies of $K_t$, which come from the $t$-cliques of $G$. 
	By property (2), these $t$-cliques are edge-disjoint.

    It remains to check that any copy $K$ of a rainbow $t$-clique in this colouring of $K_n$ is contained in $G$.
	The only edge of $K$ that might not be contained in $E(G)$ must have colour $1$. Let $a$ and $b$ be the endpoints of this edge, and let $c$ and $d$ be two other vertices of $K$ (here we use that $t \ge 4$). 
	Since triangles $acd$ and $bcd$ have no edge of colour $1$, they are triangles of $G$. 
	Property (3) implies that each of these triangles lies inside a $t$-clique of $G$, and, as they share the edge $cd$, by property (2) these $t$-cliques must coincide. 
	In particular, the edge $ab$ also lies inside the same $t$-clique of $G$.
	Thus, all the edges of $K$ are contained in $E(G)$, completing the proof of the lower bound.

	\section{Construction for the lower bound}
	\label{sec:construction}

In this section, we prove the lower bound in Theorem~\ref{th:rainbow_triangles}.  
   
First, we consider the case when $n = 2^m + m$, for some $m \in \NN$.
	Divide the vertex set of $K_n$ into two parts: $R$, indexed by the elements of $[m]:= \{1,\ldots, m\}$, and $L$, indexed by $\{0,1\}^m$, the set of 0-1 vectors of length $m$.
	Then we colour the edges of $K_n$ as follows:
	\begin{enumerate}
		\item For $u \in L = \{0,1\}^m$ and $i \in R = [m]$, colour the edge $u i$ blue if $u_i = 1$ and red otherwise.
		\item For $u, v \in L$, colour the edge $uv$ green if the Hamming distance $\sum_{i=1}^m |u_i - v_i|$ between them equals $1$.
		\item Colour all the remaining edges red.
	\end{enumerate}
	Figure~\ref{fig:lower_colouring} below illustrates this colouring for $m=2$, where there are $n = 6$ vertices.

\begin{figure}[ht]
\centering
\definecolor{qqzzqq}{rgb}{0,0.6,0}
\definecolor{ffqqqq}{rgb}{1,0,0}
\definecolor{qqqqff}{rgb}{0,0,1}
\definecolor{ududff}{rgb}{0.30196078431372547,0.30196078431372547,1}
\definecolor{cqcqcq}{rgb}{0.7529411764705882,0.7529411764705882,0.7529411764705882}
\begin{tikzpicture}[scale=0.7, line cap=round,line join=round,>=triangle 45,x=1cm,y=1cm]
\draw [color=cqcqcq, xstep=0cm,ystep=0cm] (0.8,0) grid (12,7);
\clip(0.8,0) rectangle (12,7);
\draw [line width=0.8pt,color=qqqqff] (6,5)-- (10,4);
\draw [line width=0.8pt,color=qqqqff] (6,5)-- (10,3);
\draw [line width=0.8pt,color=ffqqqq] (3,2)-- (10,4);
\draw [line width=0.8pt,color=qqqqff] (6,2)-- (10,4);
\draw [line width=0.8pt,color=qqqqff] (3,5)-- (10,3);
\draw [line width=0.8pt,color=ffqqqq] (6,2)-- (10,3);
\draw [line width=0.8pt,color=ffqqqq] (3,2)-- (10,3);
\draw [line width=0.8pt,color=ffqqqq] (3,5)-- (10,4);
\draw [line width=0.8pt,color=qqzzqq] (6,5)-- (6,2);
\draw [line width=0.8pt,color=qqzzqq] (3,2)-- (6,2);
\draw [line width=0.8pt,color=qqzzqq] (3,2)-- (3,5);
\draw [line width=0.8pt,color=qqzzqq] (3,5)-- (6,5);
\draw [rotate around={-89.1815445383102:(4.516833430415449,3.4949546368931714)},line width=0.8pt] (4.516833430415449,3.4949546368931714) ellipse (3.2950036971853027cm and 2.0922601273488284cm);
\draw [rotate around={-89.84685826518263:(9.98908049492407,3.4984536984444903)},line width=0.8pt] (9.98908049492407,3.4984536984444903) ellipse (1.5398468023609406cm and 0.7294482494371953cm);
\begin{scriptsize}
\draw [fill=black] (3,5) circle (3pt);
\draw[color=black] (2.1211983458828807,5.6271325403263495) node [thick, font=\fontsize{10}{0}\selectfont, thick] {(0,1)};
\draw [fill=black] (6,5) circle (3pt);
\draw[color=black] (7,5.6271325403263495) node [thick, font=\fontsize{10}{0}\selectfont, thick] {(1,1)};
\draw [fill=black] (6,2) circle (3pt);
\draw[color=black] (7,1.5) node [thick, font=\fontsize{10}{0}\selectfont, thick] {(1,0)};
\draw [fill=black] (3,2) circle (3pt);
\draw[color=black] (2.1211983458828807,1.5) node [thick, font=\fontsize{10}{0}\selectfont, thick] {(0,0)};
\draw [fill=black] (10,4) circle (3pt);
\draw[color=black] (10.4,4) node [thick, font=\fontsize{10}{0}\selectfont, thick] {$1$};
\draw [fill=black] (10,3) circle (3pt);
\draw[color=black] (10.4,3) node [thick, font=\fontsize{10}{0}\selectfont, thick] {$2$};
\end{scriptsize}
\end{tikzpicture}
\caption{\Small{The colouring for $m = 2$. Only the edges of rainbow triangles are shown.}}
\label{fig:lower_colouring}
\end{figure}
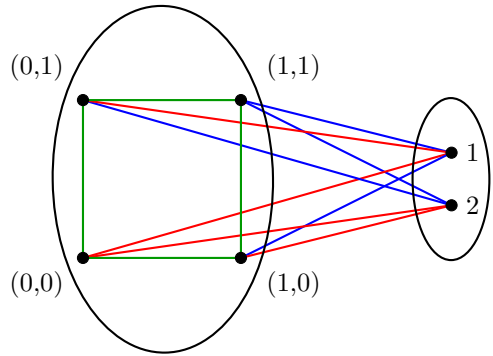

Every rainbow triangle must include a green edge, which only appear inside $L$. 
We claim that every green edge $uv$, with $u,v\in L$, is contained in exactly one rainbow triangle. Note that the third vertex of a rainbow triangle containing $u$ and $v$ must be in $R$ since $L$ contains no blue edges. 
By definition, there is exactly one vertex in $R$ that is connected to $u$ and $v$ by edges of different colours: the vertex $i\in [m]$, where $i$ is the unique coordinate in which $u$ and $v$ differ. 

It remains to show that these $m2^{m-1}$ rainbow triangles are pairwise edge-disjoint. We have already checked that they cannot share green edges. If a red or blue edge lies in two rainbow triangles, it has to connect $L$ to $R$. Let $v \in L$ and $i \in R$ be its endpoints. Then the green edges of these two rainbow triangles are adjacent to $v$ (denote these edges by $vu$ and $vu'$, where $u, u' \in L$, $u \neq u'$). But then, by the same argument as above, $v$ and $u$ differ only in the coordinate $i$, and $v$ and $u'$ differ only in the coordinate $i$, contradicting $u \neq u'$. So, the rainbow triangles cannot share red or blue edges as well.

	Therefore, our colouring is almost 3-Gallai and has exactly one rainbow triangle for each edge on the green Hamming cube in $L$, which implies that
	\[
		\tau_3(n) \ge m 2^{m-1} = \big(1/2 + o(1)\big) n \log n.
	\] 

    \medskip
	
	Now we return to the general case where $n$ is an arbitrary positive integer. 
	Let $m$ be the smallest integer such that $n \le 2^m + m$, and consider the almost 3-Gallai colouring of $K_{2^m + m}$ described above, with $V(K_{2^m + m}) = L \cup R$.
	Each vertex $u\in L = \{0,1\}^m$ may be viewed as a natural number $b(u) := \sum \limits_{i=1}^{m} 2^{i-1}u_i$, written in binary.

	Let $L' := b^{-1} ([0, n-m-1])$ be the set of vectors representing the smallest $n-m$ natural numbers. 
	The restriction of our almost 3-Gallai colouring from $L\cup R$ to $L' \cup R$ is an almost 3-Gallai colouring of $K_n$, and the number of rainbow triangles in it is equal to the number of green edges inside $L'$. 
    
    It remains to check that the number of green edges inside $L'$ is at least $(1/2-o(1))n \log n$. This follows from the standard proof of the fact that the edge-isoperimetric inequality for the hypercube is asymptotically sharp. While this result is implicit in the original works of Bernstein~\cite{bernstein1967maximally}, Harper~\cite{harper1964optimal}, Hart~\cite{hart}, and Lindsey~\cite{lindsey1964assignment}, we could not find a reference to an explicit computation. For completeness, we include it here.
    
	First, we prove an upper bound on the number $e_G(L', L\setminus L')$ of green edges between $L'$ and $L \setminus L'$ by classifying them according to the coordinate in which their endpoints differ.	
	For each $i \in [m]$, let $T_i$ be the set of green edges $u u^{(i)}$ where $u \in L'$, and the unique vertex $u^{(i)}$ that differs from $u$ only in the $i$-th coordinate is contained in the complement $L\setminus L'$. 
	As $b(L') = [0, n-m-1]$ and $\big |b(u) - b(u^{(i)}) \big | = 2^{i-1}$, we have $n-m \leq b(u^{(i)}) < n-m +2^{i-1}$ for the endpoints $u^{(i)} \in L\setminus L'$ of the edges in $T_i$. As $T_i$ is a matching, we have $|T_i| \le 2^{i-1}$, and hence we have 
	\begin{align}
	\label{eq:count_green}
		e_G(L',L\setminus L') = \sum_{i=1}^m |T_i| \le \sum_{i=1}^m 2^{i-1} = 2^{m} - 1.
	\end{align}
    Since each vertex in $L$ has green degree $m$ in our colouring of $K_{2^m+m}$, it follows from~\eqref{eq:count_green} that the number of green edges inside $L'$, and hence the number of rainbow triangles in our almost 3-Gallai colouring of $E(K_n)$, is at least
	\[
		\frac{1}{2}\big(m\cdot |L'| - e_G(L', L \setminus L') \big) \ge \frac{1}{2}\big(m(n-m) - 2^{m}\big) = \left (\dfrac{1}{2} - o(1) \right ) \cdot n \log n.
	\]

	\section{The upper bound}
	\label{sec:outline_upper}

        In this section we prove the upper bound in Theorem~\ref{th:rainbow_triangles}.
        From this point onward, we will refer to almost 3-Gallai colourings simply as almost Gallai colourings. Note that the definition of an almost Gallai colouring places no restriction on the number of colours used.
        In our first lemma, we establish that restricting ourselves to only three colours affects the maximum number of rainbow triangles by at most a constant factor.

        Let $g(n)$ denote the maximum number of rainbow triangles in an almost Gallai 3-colouring of $K_n$. 
        For a colouring $c: E(K_n)\to \mathbb{N}$, we let $\rho(c)$ denote the number of rainbow triangles induced by $c$.

	\begin{lemma}
	\label{lemma:3-colors-enough}
For every $n\in \mathbb{N}$ we have
\[
			g(n) \le \tau_{3}(n) \le \dfrac{9}{2} g(n).
		\]
	\end{lemma}
	\begin{proof}
		The lower bound $\tau_3(n) \ge g(n)$ is trivial, as in $g(n)$ we take the maximum over a smaller set. 
		For the upper bound, first note that combining an almost Gallai colouring $c: E(K_n) \to \NN$ of $K_n$ with an arbitrary map $f: \NN \to \{R, G, B\}$ gives a $3$-colouring $f \circ c: E(K_n) \to \{R, G, B\}$ of $K_n$, which still satisfies the almost Gallai property. Indeed, any triangle which is rainbow in $f \circ c$ is also rainbow in $c$. So, the rainbow triangles of $f\circ c$  are pairwise edge-disjoint, and hence $\rho(f \circ c) \le g(n)$.

    Consider an almost Gallai colouring $c$ with $\rho(c)= \tau_3(n)$, i.e. with the maximum possible number of rainbow triangles, 
    and let $f$ be the random map, which sends each element of $\NN$ independently to $R, G$ or $B$ with probability $1/3$. 
		Then each of the $\tau_3(n)$ rainbow triangles of $c$ remains rainbow in $f \circ c$ with probability $\frac{3!}{3^3} = \frac{2}{9}$. 
		Thus, by linearity of expectation, we have
		\begin{equation}
		\label{eq:circ_2}
			\Ex{\rho(f \circ c)} = \frac{2}{9}\rho(c) =\frac{2}{9}\tau_3(n).
		\end{equation}
        As $\Ex{\rho(f \circ c)} \le g(n)$, the desired upper bound on $\tau_{3}(n)$ follows.
  \end{proof}
 
 For the rest of this section, we focus on proving that $g(n) = O(n^{\sqrt{2}} \log n)$. First we introduce two key lemmas and apply them to deduce our main theorem. The proofs of the lemmas are postponed to subsequent sections.  
  We will use the three colours red (R), blue (B) and green (G).
	When the colouring is clear from the context, we will denote the red, blue and green neighbourhoods of a vertex $v$ by $N_R(v)$, $N_B(v)$ and $N_G(v)$, respectively.
	The sizes of these neighbourhoods will be denoted by $d_R(v)$, $d_B(v)$ and $d_G(v)$, respectively.

  The first lemma gives an $O(n \log n)$ upper bound on the number of rainbow triangles which are not entirely contained in any of the three monochromatic neighbourhoods of a given vertex. 
  If there is a vertex $v$ such that none of its three colour degrees is ``large'', we can apply the induction hypothesis to each of these neighbourhoods and combine it with the bound of the lemma to obtain an efficient upper bound on the total number of rainbow triangles.

	\begin{lemma}
		\label{lemma:balanced}
		There exists an absolute constant $C_1 >0$ such that the following holds for all $n \in \mathbb{N}$.
		For any almost Gallai $3$-colouring $c: E(K_n) \to \{R, B, G\}$ of $K_n$ and any vertex $v \in V(K_n)$, the number of rainbow triangles of $c$ which are not fully contained in any of the sets $N_R(v)$, $N_B(v)$ or $N_G(v)$ is at most $C_1 n \log n$.
		\end{lemma}
  The proof of this lemma appears in Section~\ref{sec:lemma-balanced}.  Lemma~\ref{lemma:balanced} is effective if there is a vertex whose neighbourhoods in each of the three colours are not very large. Otherwise, every vertex has a ``dominant colour'' on its incident edges, and the inductive bound is not sufficient. In this case, it turns out that most vertices share the same dominant colour. Our second lemma provides an estimate on the number of rainbow triangles inside a set of vertices, which contains relatively few edges in one of the colours.

	\begin{lemma}
	\label{lemma:unbalanced}
		There exists an absolute constant $C_2 >0$ such that the following holds for all sufficiently large $D$ and $n$. 
		Let $c: E(K_n) \to \{R, B, G\}$ be an almost Gallai $3$-colouring of $K_n$ such that the green degree $d_G(v)\le D$ for all	$v \in V(K_n)$. Then, the number of rainbow triangles in $c$ is at most $C_2 n \sqrt{g(D) \log n}$.
	\end{lemma}	

The proof of this lemma is postponed to Section~\ref{sec:preliminary}. 
 We now proceed to give the 
 proof of Theorem \ref{th:rainbow_triangles} assuming Lemmas \ref{lemma:balanced} and \ref{lemma:unbalanced}.

	\begin{proof}[\textbf{Proof of the upper bound in Theorem \ref{th:rainbow_triangles}}]
	 	Define $F(x) = x^{\sqrt{2}} \log(x+1)$. Then by Lemma \ref{lemma:3-colors-enough} it is sufficient to prove that $g (n) \leq C \cdot F(n)$ for some constant $C$. We show this by induction on $n$. 
 
  Let $c: E(K_n) \rightarrow \{R,B, G\}$ be an arbitrary almost Gallai $3$-colouring of $K_n$. We divide the proof into two cases based on whether every vertex has a dominant colour or not.
		Let $D = n^{2 - \sqrt{2}} \leq n$, and assume that $n$ is sufficiently large.
  
  \textbf{The balanced case.} First suppose that there is a vertex $v$ such that its degree in each colour in $c$ is at most $n-D$, that is
		\[
			\max\left\{d_R(v), d_B(v), d_G(v)\right\} \le n-D.
		\]
		Each rainbow triangle is either fully contained in one of the neighbourhoods $N_R(v), N_B(v)$ and $N_G(v)$, or it is not. Hence by Lemma~\ref{lemma:balanced} the colouring $c$ contains at most
		\[
			g(d_R(v)) + g(d_B(v)) + g(d_G(v)) + C_1 n \log n
		\]
		rainbow triangles. We apply the induction hypothesis to estimate the first three summands. Since $F$ is increasing and convex on $[0, \infty)$, we have
		\[
			\max_{x+y+z \le n, \max(x, y, z) \le n-D} F(x) + F(y) + F(z) \le F(0) + F(D) + F(n-D). 
		\]
		Considering the derivative 
		\[
		F'(x) = x^{\sqrt{2}-1}\left(\sqrt{2} \log(x+1) + (\log e)\frac{x}{x+1}\right) \ge \sqrt{2} \cdot x^{\sqrt{2}-1} \log(x+1),
		\]
		we obtain that
		\[
			F(n) - F(n - D) \ge \sqrt{2} \cdot D (n-D)^{\sqrt{2} - 1} \log (n-D+1) = (\sqrt{2} - o(1)) \cdot n \log n.
		\]
		Therefore,
		\begin{align}
			g(n) & \le C \big(F(D) + F(n-D)\big) + C_1 n \log n \nonumber \\
			& \le  C \cdot F(n) + C \cdot \left (F(n - D) + F(D) - F(n) + \frac{C_1}{C} n \log n\right )  \nonumber \\
			& \le C \cdot F(n) + C \cdot \left ( \left (\frac{C_1}{C} - \sqrt{2} + o(1) \right ) \cdot n \log n + D^{\sqrt{2}} \log(D+1) \right ). \label{eq:C_1/C}
		\end{align}

		Since $D^{\sqrt{2}} \log(D+1) \le n^{2\sqrt{2} - 2} \log n = o(n)$, one can ensure that the second summand of \eqref{eq:C_1/C} is negative by taking $C$ sufficiently large compared to $C_1$.
	
		\textbf{The unbalanced case.} Now we are in the situation when each vertex has a ``dominant'' colour: such that the degree in this colour is at least $n - D$. Depending on the dominant colour we put each vertex in one of three sets $V_R, V_G$ or $V_B$, 
		\[
			V_R \cup V_G \cup V_B = V(K_n).
		\]
		We claim that one of these three sets has to be large. Indeed, double-counting the number of red edges between $V_R$ and its complement gives
		\[
			|V_R| (n-D-|V_R|) \le Dn.
		\]
		Solving this quadratic inequality yields that
		\[
			\text{ either } |V_R| \le 2D \text{ or } |V_R| \ge n - 3D,  
		\]
		and similar bounds for $V_B$ and $V_G$. Consequently, for one of the colours the corresponding set is of size at least $n-3D$. 
		By swapping colours we may assume that $|V_R| \ge n - 3D$. 
  
  The restriction of our colouring to $V_R$ satisfies the assumption of Lemma~\ref{lemma:unbalanced} (even for both blue and green colours, but here one of them is enough). Using this lemma and the induction hypothesis, we conclude that the number of rainbow triangles of $c$ inside $V_R$ is at most  
		\begin{equation}
		\label{eq:c_V_R}		
   C_2 n \sqrt{g(D) \log n} \le C_2 \sqrt{C} \cdot n D^{\frac{\sqrt{2}}{2}} \log n.
		\end{equation}
		
		It remains to count the rainbow triangles that intersect the complement of $V_R$, which has size at most $3D$. Observe that each vertex is contained in at most $D$ rainbow triangles: indeed, each rainbow triangle includes at least one edge of a ``non-dominant'' colour of the vertex, and no such edge belongs to more than one rainbow triangle by the almost Gallai property. Thus the number of rainbow triangles not contained in $V_R$ is at most $3D^2$. Combining this with \eqref{eq:c_V_R} gives that the number of rainbow triangles in $c$ is at most
		\[
   C_2 \sqrt{C} \cdot n D^{\frac{\sqrt{2}}{2}} \log n + 3D^2 = C_2\sqrt{C} \cdot n^{\sqrt{2}} \log n + 3n^{4 - 2\sqrt{2}}.
		\]
		Taking $C$ sufficiently large with respect to $C_2$ guarantees that the last expression is at most $C n^{\sqrt{2}} \log n = C \cdot F(n)$ for every $n$. 
        \end{proof}

        \section{Preliminary lemmas and proof of Lemma~\ref{lemma:unbalanced}}
        \label{sec:preliminary}

        The key to proving Lemmas~\ref{lemma:balanced} and~\ref{lemma:unbalanced} is the ability to bound the number of rainbow triangles crossing between two disjoint sets of vertices $X$ and $Y$. In the simplest case, one colour (say, green) is completely absent from the edges between $X$ and $Y$.
In this scenario the green edges within $X$ that form a rainbow triangle with a vertex in $Y$ will turn out to exhibit a hypercube-like structure. 
For $y \in \mathbb{N}$, we say that graph $H$ is a {\em $y$-dimensional hypercube-like} graph if there exists a partition $V(H) = \bigcup_{S \se [y]} V_S$ of its vertices such that:
\begin{enumerate}
    \item [(1)] for every edge $uv \in E(H)$, there exist $S, T \se [y]$ such that $|S\bigtriangleup T| = 1$, with $u \in V_S$ and $v \in V_T$;
    \item [(2)] $H[V_S,V_T]$ is a matching (not necessarily perfect or non-empty) for all sets $S,T \se [y]$ such that $|S\bigtriangleup T|=1$.
\end{enumerate}

The well-known edge-isoperimetric inequality for the hypercube, first established in the 1960s by  Bernstein~\cite{bernstein1967maximally}, Harper~\cite{harper1964optimal}, Hart~\cite{hart} and Lindsey~\cite{lindsey1964assignment}, determines for all integers $x, y \in \mathbb{N}$ the maximum number of edges that can be induced by an $x$-element subset in the $y$-dimensional hypercube graph. The answer is always at most $\frac{1}{2}x\log x$.   
Our first lemma generalises this to hypercube-like graphs.

\begin{lemma}\label{lemma:2-colors}  For all integers $x, y \in \mathbb{N}$, every $y$-dimensional hypercube-like graph on $x$ vertices has at most $\frac{1}{2}x\log x$ edges. 
\end{lemma}

	\begin{proof}
		The proof goes by induction on $y$.
        Let $H$ be a $y$-dimensional hypercube-like graph.
        By definition, we know that there exists a partition $(V_S: S \se [y])$ of $V(H)$ satisfying items 1 and 2 above.
        Now, set $V_1 = \bigcup_{S \ni y} V_S$ and $V_2 = V(H)\setminus V_1 = \bigcup_{S \not\ni y} V_S$.
        The graphs $H[V_1]$ and $H[V_2]$ are $(y-1)$-dimensional hypercube-like, and hence we may use the induction hypothesis to obtain that
        \[e(H[V_1]) \le \dfrac{1}{2}|V_1|\log |V_1| \qquad \text{and} \qquad e(H[V_2]) \le \dfrac{1}{2}|V_2|\log |V_2|.\]
        
        By item 1, the only potential edges of $H$ between $V_1$ and $V_2$ are those between parts $V_S \se V_1$ and $V_{T} \se V_2$ such that $S = T \cup \{y\}$.
        From item 2 it follows that $H[V_1,V_2]$ is a matching, and hence the number of edges $e(H[V_1,V_2])$ is at most $\min (|V_1|,|V_2|)$.
		By summing it up, we obtain
		\[
			e(H) \le \frac{1}{2} \big( |V_1| \log |V_1| + |V_2| \log |V_2| + 2\min(|V_1|, |V_2|) \big).
		\]
This is at most $2^{-1}(|V_1| + |V_2|) \log (|V_1| + |V_2|)$ by the following elementary inequality applied with $\lambda = \frac{\min(|V_1|, |V_2|)}{|V_1| + |V_2|}$:
\[
    2\lambda  \le - \lambda \log(\lambda) - (1-\lambda) \log(1-\lambda) \text{ for } \lambda \in [0, 1/2].
\]
To prove this inequality, note that it holds for $\lambda = 0$, $\lambda = 1/2$, and that the right-hand side is concave.
	\end{proof}

	Our next lemma bounds the number of rainbow triangles crossing between two disjoint sets of vertices $X$ and $Y$, also addressing the situation when some green edges do appear between $X$ and $Y$.
    For an almost Gallai colouring of $K_n$, disjoint sets $X,Y \se V(K_n)$ and a colour $q \in \{R,G,B\}$, we denote by $\tau^{\mathrm{q}}_{\mathrm{rbw}}(X, Y)$ the number of rainbow triangles $yx_1x_2$ with $y \in Y$, $x_1, x_2 \in X$ and $x_1 x_2$ coloured in $q$ (the subscript ``rbw'' stands for ``rainbow'').
	\begin{lemma}
	\label{lemma:dnlog}
	Let $d \in \mathbb{N}_{\ge 0}$, let $X$ and $Y$ be disjoint subsets of $V(K_n)$, and let $c:E(K_n)\to \{R,G,B\}$ be an almost Gallai colouring of $K_n$ such that $|N_G(y) \cap X| \le d$ for all $y \in Y$.
	Then, we have 
	\begin{align*}
    \tau^{\mathrm{G}}_{\mathrm{rbw}}(X, Y) \le \frac{e}{2} (d+1) |X| \log|X|.
	\end{align*}
	\end{lemma}
	\begin{proof} 

        First, we consider the case $d = 0$, when there are no green edges between $X$ and $Y$. 
            
        Let $E'$ be the set of those green edges in $X$ which form a rainbow triangle with a vertex from $Y$. By the almost Gallai-property we have $\tau_{\mathrm{rbw}}^{\mathrm{G}} (X,Y) = |E'|$. Our plan is to show that the graph $(X,E')$ is a $|Y|$-dimensional hypercube-like graph and conclude by the previous lemma that $|E'| \leq \frac{1}{2}|X|\log|X|$. 
           
We partition $X$ according to the blue neighborhoods of its vertices in $Y$. For each set $S \se Y$, define $X_S := \{ v \in X: N_B(v)\cap Y = S \}$, so $X=\bigcup_{S \se Y} X_S$. 
       
Since every edge between $X$ and $Y$ is red or blue, for any green edge $x_1x_2$ in $X$ and a vertex $y\in Y$ the triangle $yx_1x_2$ is rainbow if and only if $y$ is contained in the symmetric difference of the blue neighborhoods $N_B(x_1)$ and $N_B(x_2)$ of the endpoints of the green edge. 
Hence, by the almost Gallai property, for any green edge $x_1x_2$ in $X$ the sets $N_B(x_1)\cap Y $ and $N_B(x_2)\cap Y$ are either equal (in this case the green edge does not participate in any $X,Y$-crossing rainbow triangle) or one of them is the other with one vertex removed.
So, any edge in $E'$ must connect two sets $X_S$ and $X_T$ such that $|S\bigtriangleup T| = 1$, as required in a hypercube-like graph. 

We still need to check that $E'$ forms a matching between two such sets $X_S$ and $X_T$. A green edge $x_1x_2$ with $x_1\in X_S$ and $x_2\in X_T$ is in a rainbow triangle with the unique vertex $y\in Y$: the one in $S \bigtriangleup T$.
However, for every vertex $y\in Y$ the green edges in $X$ which form a rainbow triangle with $y$ must form a matching: otherwise, if two such green edges shared a common vertex, the edge from $y$ to this vertex would be part of two rainbow triangles, contradicting the almost Gallai property of our colouring.
So, $(X,E')$ is indeed a hypercube-like graph, and by  Lemma~\ref{lemma:2-colors} we have
        \begin{align*}
        \tau^{\mathrm{G}}_{\mathrm{rbw}}(X, Y) = |E'| \le \frac{1}{2}|X|\log |X|.
        \end{align*}

		From now on, we assume that $d \ge 1$. 
        Let $S_X$ be a $1/(d+1)$-random subset of $X$ and set $S_Y \subset Y$ to be the set of vertices not connected to $S_X$ by a green edge.
		As all edges between $S_X$ and $S_Y$ are red or blue, it follows from the $d = 0$ case that
        \[
		\tau^{\mathrm{G}}_{\mathrm{rbw}}(S_X, S_Y) \le \dfrac{1}{2} |S_X| \log |S_X| \le \frac{1}{2} |S_X| \log |X|.
		\]
		By linearity of expectation, we have
		\begin{equation}
		\label{eq:X_A,X_B-upper}
	    \Ex{\tau^{\mathrm{G}}_{\mathrm{rbw}}(S_X, S_Y)} \le \frac{1}{2}\log |X| \cdot \Ex{|S_X|} = \frac{|X| \log |X|}{2(d+1)}.
		\end{equation}		
		On the other hand, 
		for a fixed rainbow triangle $yx_1x_2$ with $y \in Y$, $x_1, x_2 \in X$ and $x_1 x_2$ coloured green we have
		\begin{align*}
			\PP(x_1, x_2 \in S_X, \; y \in S_Y) & = 
			\PP(x_1, x_2 \in S_X, \; N_G(y)\, \cap \, S_X   = \emptyset) \\
			& \ge \dfrac{1}{(d+1)^2} \cdot \left (1 - \dfrac{1}{d+1} \right )^d \\
			& \ge \frac{1}{(d+1)^2}\cdot\frac{1}{e}. 
		\end{align*}
        In the first inequality above we used that $|N_G(y) \cap X| \le d$. The second inequality follows from the fact that $(1+1/d)^d \le e$ for all $d \in \mathbb{N}$. Therefore, 
		\begin{equation}
		\label{eq:X_A,X_B-lower}
			\Ex{\tau^{\mathrm{G}}_{\mathrm{rbw}}(S_X, S_Y)} \ge \frac{1}{e(d+1)^2} \cdot \tau^{\mathrm{G}}_{\mathrm{rbw}}(X, Y).
		\end{equation}
		By combining~\eqref{eq:X_A,X_B-upper} and~\eqref{eq:X_A,X_B-lower}, we conclude the proof.
	\end{proof}

\begin{proof}[\textbf{Proof of Lemma~\ref{lemma:unbalanced}}]

    Since the green degree of each vertex is bounded by $D$, the number of rainbow triangles in $c$ is at most $nD$. As $g(D) = \Omega(D)$, if $g(D) < \log n$ then we automatically have
    \[
    nD = O(n \sqrt{g(D)\log n}).
    \]
    Thus, we may assume that $g(D) \ge \log n$.

We count the rainbow triangles in $c$ based on the vertex incident to their red and blue edges. 
Let $c_{RB}(v)$ be the number of rainbow triangles $vxy$ such that $vx$ is red and $vy$ is blue. 
Recall that $\rho(c)$ denotes the number of rainbow triangles in $c$.
Our goal is to obtain an upper bound on
\[
\sum_{v \in V(K_n)} c_{RB}(v) = \rho(c). 
\]
		
	We say that a quadruple $(v_G, v_R, v_B, u)$ of pairwise distinct vertices of $K_n$ is \textit{nice} if the edge $v_G v_R$ is blue, the edge $v_G v_B$ is red, and the edges $v_R v_B$ and $v_G u$ are green. 

    \begin{figure}[ht]
    \centering
    \definecolor{qqzzqq}{rgb}{0,0.6,0}
    \definecolor{ffqqqq}{rgb}{1,0,0}
    \definecolor{qqqqff}{rgb}{0,0,1}
    \begin{tikzpicture}[scale=0.7, line cap=round,line join=round,>=triangle 45,x=1cm,y=1cm]
    \draw [line width=0.8pt,color=qqqqff] (1.5, 1.5) -- (3, 0.3);
    \draw [line width=0.8pt,color=ffqqqq] (1.5, 1.5) -- (3, 2.7);
    \draw [line width=0.8pt,color=qqzzqq] (-0.5, 1.5) -- (1.5, 1.5);
    \draw [line width=0.8pt,color=qqzzqq] (3, 0.3) -- (3, 2.7);
    \begin{scriptsize}
    \draw [fill=black] (-0.5, 1.5) circle (2pt);
    \draw[color=black] (-0.5, 1.1) node {$u$};
    \draw [fill=black] (1.5, 1.5) circle (2pt);
    \draw[color=black] (1.5, 1.1) node {$v_G$};
    \draw [fill=black] (3, 0.3) circle (2pt);
    \draw[color=black] (3.5, 0.2) node {$v_R$};
    \draw [fill=black] (3, 2.7) circle (2pt);
    \draw[color=black] (3.5, 2.5) node {$v_B$};
    \end{scriptsize}
    \end{tikzpicture}
    \caption{\Small{A nice quadruple of vertices.}}
    \end{figure}
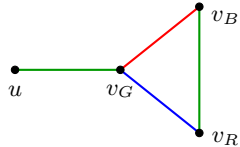
    To estimate the number of rainbow triangles in $K_n$ 
    we use a double-counting argument on the number $X$ of nice quadruples in $c$.
    As the number of nice quadruples $(v_G, v_R, v_B, u)$ with $v_G = v$ is $d_G(v) c_{RB}(v)$, we have
    \begin{align}\label{eq:X-2}
		X = \sum_{v \in V(K_n)} d_G(v) c_{RB}(v).
    \end{align}
    Another way to estimate $X$ is to count the number of nice quadruples $(v_G, v_R, v_B, u)$ with fixed $u$, and then sum over all vertices $u$.
    Observe that for any nice quadruple $(v_G, v_R, v_B, u)$ there are only three options for the colours of $u v_R$ and $u v_B$, which do not create a rainbow triangle sharing an edge with $v_R v_B v_G$:
	\begin{enumerate}
		\item $uv_R$ is green and $uv_B$ is red;
		\item $uv_R$ is blue and $uv_B$ is green;
		\item $uv_R$ and $uv_B$ are both green.
	\end{enumerate}
        For a fixed vertex $u$, the number of nice quadruples of type (1) is equal to $\tau^{\mathrm{B}}_{\mathrm{rbw}}(N_G(u), N_R(u))$.
        As the blue edges between $N_G(u)$ and $N_R(u)$ form a matching, we can apply Lemma~\ref{lemma:dnlog} with $d=1$, $X = N_G(u)$ and $Y = N_R(u)$ (switching the colours green and blue).
        Thus, the number of nice quadruples of type (1) is at most $4 d_G(u) \log d_G(u)$.  
        Similarly, the number of nice quadruples of type (2) with a fixed vertex $u$ is equal to $\tau^{\mathrm{R}}_{\mathrm{rbw}}(N_G(u), N_B(u))$, and by Lemma~\ref{lemma:dnlog} this is at most $4d_G(u) \log d_G(u)$.

        In a nice quadruple of type (3) the rainbow triangle is fully contained in $N_G(u)$, and hence their number is at most $g(d_G(u))$. Summing it up, we obtain
	\begin{align}\label{eq:X-1}
		X \le \sum_{u \in V(K_n)} \left(8 d_G(u) \log d_G(u)  + g(d_G(u))\right) \le 2^6n\cdot g(D).
	\end{align}
        Here we used the lower bound in Theorem~\ref{th:rainbow_triangles} (for a sufficiently large $n$) proved in Section~\ref{sec:construction}, and the fact that the green degrees are at most $D$.
	
	Let $U$ be the subset of vertices of $K_n$ with green degree at most $d := \sqrt{g(D) / \log n } \ge 1$.
    We consider two cases:

	\textbf{Case 1.} $\sum_{v \notin U} c_{RB}(v) \ge \rho(c) / 2$.
	
	In this case, we have
	\begin{align*}
		\rho(c) \le 
		2 \sum_{v \notin U} c_{RB}(v) \le 
		\dfrac{2}{d} \cdot \sum_{v \notin U} d_G(v) c_{RB}(v) \le \dfrac{2X}{d} \le  \dfrac{2^7ng(D)}{d} 
= 2^7n \sqrt{g(D) \log n},
	\end{align*}
	where we used~\eqref{eq:X-2}, \eqref{eq:X-1} and the definition of $d$.

    \textbf{Case 2.} $\sum_{v \in U} c_{RB}(v) \ge \rho(c) / 2$.

    The sum $\sum_{v \in U} c_{RB}(v)$ counts the number of rainbow triangles $vxy$ with $v \in U$ and edge $xy$ coloured green. First, we estimate the number of such triangles where the edge $xy$ touches $U$. Since each vertex in $U$ is adjacent to at most $d$ green edges, there are at most $d|U|$ such triangles. 
    Next, we estimate the number of such triangles with $x,y \notin U$. As $d_G(u) \le d$ for all $u \in U$, we can apply Lemma \ref{lemma:dnlog} with $X = V \setminus U$ and $Y = U$, which gives an upper bound of $2(d + 1)n\log n$.

    Summing it up, we obtain
    \[
    \rho(c) \le 2 \sum_{v \in U} c_{RB}(v) \le 2d|U| + 4(d + 1)n \log n \le (6d+4)n \log n.
    \]
    As $d = \sqrt{g(D) / \log n}$, this concludes the proof.
    \end{proof}

    \section{Proof of Lemma~\ref{lemma:balanced}}
    \label{sec:lemma-balanced}

    We say that a 2-colouring of a complete tripartite graph is \emph{good} if no two monochromatic triangles share an edge.
    Let $\gamma(n)$ be the maximum number of monochromatic triangles in a good red-blue colouring of a complete tripartite graph on $n$ vertices.
    The following lemma plays a key role in the proof of Lemma~\ref{lemma:balanced}.

	\begin{lemma}
	\label{lemma:3-parts-2-colors}
	For all $n \ge 3$ we have
	\[ \gamma(n) \le 4n-10.\]
	\end{lemma}	

    This result is sharp up to a multiplicative constant factor. 
    Indeed, consider a complete tripartite graph $K(V_1,V_2,V_3)$ on $n$ vertices where $|V_1| = 1$ and $\big||V_2| - |V_3|\big| \le 1$.
    Colour all the edges between $V_1$ and $V_2 \cup V_3$ red. Then, colour the edges of some maximum matching between $V_2$ and $V_3$ red, leaving the remaining edges blue.
    The resulting colouring is good and contains $\big \lfloor \frac{n-1}{2} \big \rfloor$ red triangles.

    \begin{proof}[\textbf{Proof of Lemma \ref{lemma:balanced}, assuming Lemma~\ref{lemma:3-parts-2-colors}}]
    Fix an arbitrary vertex $v \in V(K_n)$.
    We classify the rainbow triangles $T$ which are not fully contained in any of the neighbourhoods $N_R(v)$, $N_B(v)$ or $N_G(v)$:
	\begin{enumerate}
		\item[$\,$] \hspace{-26pt}Type 1: $T$ contains $v$; 
		\item[$\,$]\hspace{-26pt}Type 2: one of $N_R(v)$, $N_B(v)$, $N_G(v)$ contains two vertices of $T$; 
  
		\item[$\,$]\hspace{-26pt}Type 3: each of $N_R(v)$, $N_B(v)$, $N_G(v)$ contains one vertex of $T$.
  \end{enumerate}
Define {\em special edges} as the edges which are not adjacent to $v$ but form a rainbow triangle with $v$.
Special edges form a matching: otherwise, the edge from $v$ to the common vertex of two special edges would participate in two rainbow triangles, contradicting the almost Gallai property of our colouring. In particular, the number of rainbow triangles of Type 1 is at most $n/2$.
 
So, the green edges between $N_R(v)$ and $N_B(v)$, the red edges between $N_B(v)$ and $N_G(v)$, and the blue edges between $N_R(v)$ and $N_G(v)$ are special and form a matching.
Applying Lemma~\ref{lemma:dnlog} with $d=1$ to each of these three bipartite graphs, we conclude that the number of rainbow triangles of Type 2 is $O(n \log n)$.
To estimate the number of triangles of Type 3 we combine Lemma~\ref{lemma:3-parts-2-colors} with a probabilistic argument, similar to the one used in the proof of Lemma \ref{lemma:dnlog}.

We pick a random subset $X$ of $N_R(v) \cup N_B(v) \cup N_G(v)$ by deleting exactly one of the endpoints (chosen uniformly at random) of each special edge.
Set $X_R = X \cap N_R(v)$, $X_B = X \cap N_B(v)$ and $X_G = X \cap N_G(v)$, and let $T= K_n[X_R, X_B, X_G]$ be the complete tripartite subgraph of $K_n$ induced by $X_R$, $X_B$ and $X_G$.

By construction, $T$ contains no special edges. 
So, there is no green edge between $X_R$ and $X_B$, no red edge between $X_B$ and $X_G$, and no blue edge between $X_G$ and $X_R$.  
Consequently, there are only two kinds of rainbow triangles in $T$. If its vertices are $v_R\in X_R, v_B\in X_B, v_G\in X_G$, then either $v_Rv_B$ is red, $v_Bv_G$ is blue, and $v_Gv_R$ is green, or $v_Rv_B$ is blue, $v_Bv_G$ is green and $v_Gv_R$ is red.

Now, we introduce a $2$-colouring $f: E(T) \to \{1,2\}$ such that the $f$-monochromatic triangles are exactly the rainbow triangles in the original colouring of $T$. 
Red and blue edges between $X_R$ and $X_B$ receive colours $1$ and $2$, respectively. 
Blue and green edges between $X_B$ and $X_G$ receive colours $1$ and $2$, respectively. 
Finally, green and red edges between $X_G$ and $X_R$ receive colours $1$ and $2$, respectively. 
Then a triangle is rainbow in $T$ with respect to the initial colouring if and only if it is monochromatic with respect to the new one. So, we are in the setting of Lemma~\ref{lemma:3-parts-2-colors}, which allows us to conclude that the number of rainbow triangles in $T$ (denoted by $\tau_{\mathrm{rbw}}(X_R, X_B, X_G)$) satisfies 
	\begin{align*}
		\tau_{\mathrm{rbw}}(X_R, X_B, X_G) = \# \{ \text{$f$-monochromatic } \triangle'\text{s in } T \} \le 4(|X_R|+|X_B|+|X_G|) - 10 \le 4n.
	\end{align*}
	
On the other hand, we show that a random choice of $X$ will keep a constant proportion of rainbow triangles of Type 3. Note that rainbow triangles of Type 3 contain no special edges, as every special edge already lies in a rainbow triangle of Type 1. Hence each vertex of a Type 3 rainbow triangle appears in $X$ independently with probability at least $1/2$ (in fact, either $1/2$ or $1$).
Therefore, a rainbow triangle of Type 3 in $K_n$ becomes a rainbow triangle in $T$ with probability at least $1/8$. 
Therefore, the total number of rainbow triangles of Type 3 at most
\begin{align*}
8 \cdot \Ex{\tau_{\mathrm{rbw}}(X_R, X_B, X_G)} \le 32n.
\end{align*}
Summing the bounds obtained for each type of rainbow triangles, we conclude the proof.
\end{proof}	
	
    \begin{proof}[\textbf{Proof of Lemma~\ref{lemma:3-parts-2-colors}}]

        By symmetry, it is sufficient to show that the number of red triangles is at most $2n-5$.
        Let $n$ be the smallest integer such that there exists a complete tripartite graph $G=G[V_1, V_2, V_3]$ on $n$ vertices with a good red-blue colouring such that the number of red triangles in it is at least $2n-4$.
		Since the lemma holds trivially for $n=3$, we assume that $n \ge 4$.
		We can also assume that every vertex in $G$ is contained in at least three red triangles.
		Indeed, if some vertex $x \in V(G)$ is contained in at most two red triangles, then removing $x$ from $G$ results in a smaller counterexample, contradicting the choice of $n$.
		
		Let $v$ be the vertex contained in the minimal number of red triangles. 
		Without loss of generality, we may assume that $v \in V_1$. For $i \in \{2,3\}$ we set
		\[
			V_{i, R} = N_R(v) \cap V_i \qquad \text{and} \qquad V_{i, B} = N_B(v) \cap V_i
		\]		
		to be the red and blue neighbourhoods of $v$ inside $V_i$, respectively.

\begin{claim}\label{claim:partition}
	For any vertex $u \in V_1 \setminus v$, either all edges connecting $u$ to $V_{2, B}$ are blue, or all edges connecting $u$ to $V_{3, B}$ are blue. 
\end{claim}

\begin{proof}
	Suppose for contradiction that both $N_R(u) \cap V_{2, B}$ and $N_R(u) \cap V_{3, B}$ are non-empty.

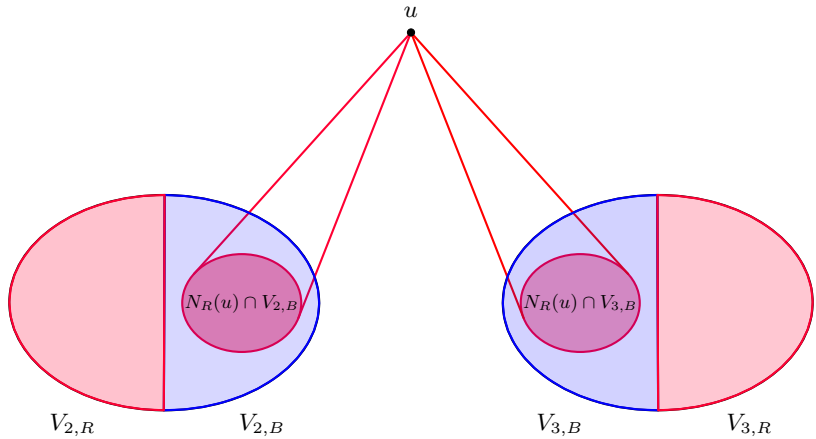
\begin{figure}[ht]
\centering
\definecolor{qqqqff}{rgb}{0,0,1}
\definecolor{ffqqqq}{rgb}{1,0,0}
\definecolor{ffqqtt}{rgb}{1,0,0.2}
\begin{tikzpicture}[line cap=round,line join=round,>=triangle 45,x=1cm,y=1cm,scale=0.9]
\clip(-6,-2) rectangle (6,4.5);
\draw [rotate around={-0.20027097795347137:(-3.644776069836047,0.005749153887732918)},line width=0.8pt] (-3.644776069836047,0.005749153887732918) ellipse (2.2888147685015383cm and 1.5916506123628433cm);
\draw [rotate around={-179.79972902204653:(3.644776069836047,0.005749153887732472)},line width=0.8pt] (3.644776069836047,0.005749153887732472) ellipse (2.2888147685015383cm and 1.5916506123628433cm);
\draw [rotate around={0:(-2.5,0)},line width=0.8pt,color=ffqqtt,fill=ffqqtt,fill opacity=0.42] (-2.5,0) ellipse (0.8791595855518389cm and 0.7231331667595404cm);
\draw [rotate around={180:(2.5,0)},line width=0.8pt,color=ffqqtt,fill=ffqqtt,fill opacity=0.35] (2.5,0) ellipse (0.8791595855518389cm and 0.7231331667595404cm);
\draw [line width=0.8pt,color=ffqqtt] (0,4)-- (-3.207091924345599,0.4297206436284907);
\draw [line width=0.8pt,color=ffqqtt] (0,4)-- (-1.663654081950697,-0.222914253702564);
\draw [line width=0.8pt,color=ffqqqq] (0,4)-- (1.663654081950697,-0.2229142537025642);
\draw [line width=0.8pt,color=ffqqqq] (0,4)-- (3.207091924345599,0.42972064362849033);
\draw [line width=0.8pt] (-3.639212644775477,1.5973900430761863)-- (-3.650339494896618,-1.5858917353007205);
\draw [line width=0.8pt] (3.639212644775477,1.5973900430761858)-- (3.650339494896618,-1.585891735300721);
\draw [shift={(-3.644776069836047,0.005749153887732918)},line width=0.8pt,color=qqqqff,fill=qqqqff,fill opacity=0.16]  (0,0) --  plot[domain=-1.5707963267948966:1.5707963267948966,variable=\t]({0.9999938911377191*2.2888147685015383*cos(\t r)+0.003495380843859593*1.5916506123628433*sin(\t r)},{-0.003495380843859593*2.2888147685015383*cos(\t r)+0.9999938911377191*1.5916506123628433*sin(\t r)}) -- cycle ;
\draw [shift={(3.644776069836047,0.005749153887732472)},line width=0.8pt,color=qqqqff,fill=qqqqff,fill opacity=0.18]  (0,0) --  plot[domain=-1.5707963267948966:1.5707963267948966,variable=\t]({-0.9999938911377191*2.2888147685015383*cos(\t r)+-0.0034953808438594708*1.5916506123628433*sin(\t r)},{-0.0034953808438594708*2.2888147685015383*cos(\t r)+0.9999938911377191*1.5916506123628433*sin(\t r)}) -- cycle ;
\draw [shift={(-3.644776069836047,0.005749153887732918)},line width=0.8pt,color=ffqqtt,fill=ffqqtt,fill opacity=0.24]  (0,0) --  plot[domain=1.5707963267948966:4.71238898038469,variable=\t]({0.9999938911377191*2.2888147685015383*cos(\t r)+0.003495380843859593*1.5916506123628433*sin(\t r)},{-0.003495380843859593*2.2888147685015383*cos(\t r)+0.9999938911377191*1.5916506123628433*sin(\t r)}) -- cycle ;
\draw [shift={(3.644776069836047,0.005749153887732472)},line width=0.8pt,color=ffqqtt,fill=ffqqtt,fill opacity=0.22]  (0,0) --  plot[domain=1.5707963267948966:4.71238898038469,variable=\t]({-0.9999938911377191*2.2888147685015383*cos(\t r)+-0.0034953808438594708*1.5916506123628433*sin(\t r)},{-0.0034953808438594708*2.2888147685015383*cos(\t r)+0.9999938911377191*1.5916506123628433*sin(\t r)}) -- cycle ;
\begin{scriptsize}
\draw[color=black] (-2.5,0) node {\Scale[0.8]{N_R(u) \cap V_{2,B}}};
\draw[color=black] (2.5,0) node {\Scale[0.8]{N_R(u) \cap V_{3,B}}};
\draw [fill=black] (0,4) circle (1.5pt);
\draw[color=black] (0,4.3) node {$u$};
\draw[color=black] (-2.2,-1.8) node {$V_{2,B}$};
\draw[color=black] (2.2,-1.8) node {$V_{3,B}$};
\draw[color=black] (-5,-1.8) node {$V_{2,R}$};
\draw[color=black] (5,-1.8) node {$V_{3,R}$};
\end{scriptsize}
\end{tikzpicture}
\caption{\Small{Partition of $V_2$ and $V_3$ into the red and blue neighbourhoods of $v$, \\ \hspace*{60pt} and the red neighbourhood of $u$ inside $V_{2,B}$ and $V_{3,B}$.}}
\label{fig:neighbourhoods}
\end{figure}

	Observe that the bipartite graph on $\big ( N_R(u) \cap V_{2, B} \big ) \cup \big ( N_R(u) \cap V_{3, B} \big )$ cannot contain a monochromatic cherry (that is, a monochromatic path with two edges).
	Indeed, suppose that $xzy$ is a monochromatic cherry with $x,y \in N_R(u) \cap V_{2, B}$ and $z \in N_R(u) \cap V_{3, B}$ (the case when $x,y \in N_R(u) \cap V_{3, B}$ and $z \in N_R(u) \cap V_{2, B}$ is completely analogous).
	If this cherry is red, then we have two red triangles $xzu$ and $yzu$ sharing an edge, which is a contradiction.
	If this cherry is blue, then we have two blue triangles $xzv$ and $yzv$ sharing an edge, again a contradiction.
	
	As a consequence of the set $\big ( N_R(u) \cap V_{2, B} \big ) \cup \big ( N_R(u) \cap V_{3, B} \big )$ not spanning a monochromatic cherry, we have
	\[\max \{ |N_R(u) \cap V_{2, B}|, |N_R(u) \cap V_{3, B}| \} \le 2.
	\]

	Now, we observe that every red triangle containing $u$ has one vertex in $N_R(u) \cap V_{2, B}$ or one vertex in $N_R(u) \cap V_{3, B}$,
    because all edges in $N_R(u) \cap N_R(v)$ must be blue.
	Moreover, as the monochromatic triangles are all edge-disjoint, there is an injective mapping from the red triangles containing $u$ to the vertices in $\big ( N_R(u) \cap V_{2, B} \big ) \cup \big ( N_R(u) \cap V_{3, B} \big )$.
	As every vertex is contained in at least three red triangles (by assumption of the minimum counterexample), we must have 
	\[ 3 \le \# \{ \triangle \text{'s containing } u \} \le |N_R(u) \cap V_{2, B}| + |N_R(u) \cap V_{3, B}| \le 4.\]
	There are two cases to be analysed:
	
	\begin{itemize}
		\item[(a)] If $|N_R(u) \cap V_{2, B}| = |N_R(u) \cap V_{3, B}| = 2$, then we must have a red matching of size two between $N_R(u) \cap V_{2, B}$ and $N_R(u) \cap V_{3, B}$:
		otherwise, we would have a blue cherry between these two sets.
		As each vertex in $N_R(u) \cap V_{2, B}$ and in $N_R(u) \cap V_{3, B}$ lies in at most one red triangle containing $u$, it follows that $u$ is contained in exactly two red triangles. This is a contradiction.
		
		\item[(b)] If $|N_R(u) \cap V_{2, B}| = 1$ and $|N_R(u) \cap V_{3, B}| = 2$ (or vice versa), then we must have a red edge $xy$ between $N_R(u) \cap V_{2, B}$ and $N_R(u) \cap V_{3, B}$.
		Any other red triangle containing $u$ cannot have a vertex in common with $xy$, but it still must have a vertex in $N_R(u) \cap V_{2, B}$ or in $N_R(u) \cap V_{3, B}$. Therefore, $u$ is contained in at most two red triangles, again giving a contradiction.
	\end{itemize}
	This proves the claim.
\end{proof}

		By Claim~\ref{claim:partition}, we may partition $V_1 \setminus v$ into two disjoint sets $U_2$ and $U_3$ in such a way that all edges between $U_2$ and $V_{2, B}$ are blue, and all edges between $U_3$ and $V_{3, B}$ are blue.
		We claim that every red triangle in $G \setminus v$ fully lies in either $U_2 \cup V_{2, R} \cup V_{3, B}$ or $U_3 \cup V_{2, B} \cup V_{3, R}$.
		Indeed, suppose that $xyz$ is a red triangle with $x \in V_1 \setminus v$, $y \in V_2$ and $z \in V_3$.
		If $x \in U_2$ then $y \in V_{2, R}$. Now, $z \in V_{3, R}$ would imply that $xyz$ and $vyz$ are two red triangles sharing an edge, which is a contradiction. Thus, we must have 
		$z \in V_{3, B}$.
		Similarly, if $x \in U_3$ then $z \in V_{3, R}$ and, as in the previous case, $y \in V_{2, B}$.

		Since the sizes of the sets $U_2 \cup V_{2, R} \cup V_{3, B}$ and $U_3 \cup V_{2, B} \cup V_{3, R}$ are less than $n$ and at least 3 (since $v$ is contained in at least 3 red triangles), we know that the theorem holds for the two tripartite graphs induced by them. Set
		\[
		m = |U_2 \cup V_{2, R} \cup V_{3, B}| \qquad \text{and} \qquad n-m-1 = |U_3 \cup V_{2, B} \cup V_{3, R}|.
		\]
		Denote by $X \ge 2n - 4$ the number of red triangles in $G$.
		As $v$ is contained in the minimal number of red triangles, by the pigeonhole principle $v$ is contained in at most $\frac{3X}{n}$ red triangles.	
		Then the total number of red triangles $X$ satisfies
		\[
		X \le (2m - 5) + (2(n-m-1) - 5) + \frac{3X}{n} = 2n - 12 + \frac{3X}{n}.
		\]
		As $X \ge 2n - 4$, we have
		\[
		2n^2-12n \ge (n-3)X \ge (n-3)(2n-4) = 2n^2-10n+12,
		\]
		which is a contradiction.
        \end{proof}

	\section{Proof of Theorem~\ref{th:characteristic_function}}\label{sec:characteristic_function}
	
	We shall use Theorems~\ref{th:rainbow_triangles} and~\ref{th:at_least_4} combined with a decoupling trick of Berkowitz~\cite{berkowitz2018local} to bound the characteristic functions of clique counts. 
	Before doing so, we need some preparation. 

	Let $n \in \mathbb{N}$, $p \in [0,1]$ and let $G^0$ and $G^1$ be two independent copies of $G(n,p)$.
    For a non-empty set $S \se E(K_n)$ and $i \in \{0,1\}$, we denote by $X_{S}^i$ the random vector $(\mathds{1}_{e \in G^{i}}: e \in S)$.
	Let $P$ be a polynomial on variables $(y_e:e \in E(K_n))$, $v \in \{0,1\}^k$ be a vector and
	$\bB = (B_0,\ldots, B_k)$ be a partition of the edges of $K_n$.
	We denote by 
	$P(X_{B_0}^0,X_{B_1}^{v_1}\ldots,X_{B_k}^{v_k})$ the corresponding random polynomial where each variable $y_e$ with $e \in B_0$ is replaced by $\mathds{1}_{e \in G^{0}}$ and each variable $y_e$ with $e \in B_i$ and $i \ge 1$ is replaced by $\mathds{1}_{e \in G^{v_i}}$.

	Given a partition $\bB = (B_0,\ldots, B_k)$ on $E(K_n)$, the $\alpha$ operator is defined as
	\begin{align*}
		\alpha_{\bB}(P) := \sum_{v \in \{0,1\}^{k}} (-1)^{|v|} P(X_{B_0}^0,X_{B_1}^{v_1}\ldots,X_{B_k}^{v_k}),
	\end{align*}
	where $|v|:= v_1+ \ldots + v_k$.

	The following theorem was proved by Berkowitz~\cite[Lemma 6]{berkowitz2018local}.
	His statement is more general, but we state it here in the form we need.

	\begin{theorem}\label{th:decoupling}
		Let $n \in \mathbb{N}$, $p \in [0,1]$ and let $G^0$ and $G^1$ be two independent copies of $G(n,p)$.
		Let $\bB = (B_0,\ldots, B_k)$ be a partition of $E(K_n)$ and let $P$ be a polynomial on variables $(y_e:e \in E(K_n))$.
		Let $X^0 = (\mathds{1}_{e \in G^{0}}:e \in E(K_n))$ and set $\varphi_{P}(s) = \mathbb{E}_{X^0}(e^{isP(X^0)})$ to be the characteristic function of $P$. 
		Then, we have
		\begin{align*}
			|\varphi_P(s)|^{2^k} \le \mathbb{E}_{Y} \left | \mathbb{E}_{X} e^{is\alpha_{\bB}(P)}  \right |,
		\end{align*}
		for all $s \in \mathbb{R}$, where $X = X^0_{B_0}$ and $Y = (\mathds{1}_{e \in G^{i}}: e \in E(K_n) \setminus B_0, i \in \{0,1\})$.
	\end{theorem}

	Now we are ready to prove Theorem~\ref{th:characteristic_function}.
	For $t \ge 3$, let $\K_t$ be the collection of all copies of $K_t$ in $K_n$ and let $X_t$ be the polynomial counting the number of copies of $K_t$ in $K_n$.
	For a partition $\bB = (B_0,\ldots, B_k)$ on $E(K_n)$ and a clique $K \in \K_t$, let $B_i^K$ be the set of edges of $K$ which are in $B_i$.
	For a set $S \se E(K_n)$ and $i \in \{0,1\}$, we set $x_{S}^i:=\prod_{e \in S} \mathds{1}_{e \in G^{i}}$.
	As $\alpha$ is a linear operator, we have
	\begin{align}\label{eq:alpha}
		\alpha_{\bB}(X_t) & = \sum_{K \in \K_t} \sum_{v \in \{0,1\}^{k}} (-1)^{|v|} x_{B_0^K}^0 x_{B_1^K}^{v_1} \cdots x_{B_k^K}^{v_k} \nonumber \\
		& = \sum_{K \in \K_t} x_{B_0^K}^0 \left (x_{B_1^K}^0 - x_{B_1^K}^1 \right ) \cdots \left ( x_{B_k^K}^0 - x_{B_k^K}^1 \right ).
	\end{align}

	From now on we fix $t \ge 3$ and a partition $\bB = (B_0,\ldots, B_k)$ on $E(K_n)$ with $k = \binom{t}{2}-1$ which generates an almost $t$-Gallai colouring of $K_n$ given by the lower bound in Theorem~\ref{th:rainbow_triangles} for $t = 3$ and Theorem~\ref{th:at_least_4} for $t \ge 4$ (we view edges in $B_i$ as edges of colour $i$). 
	Then, $\bB$ contains at least $(1/2-o(1))n\log n$ rainbow $t$-cliques when $t = 3$ and at least $n^{2-o(1)}$ rainbow $t$-cliques when $t \ge 4$.

	As $x_{\emptyset}^0 = x_{\emptyset}^1 = 1$, we have that the product $(x_{B_1^K}^0 - x_{B_1^K}^1) \cdots (x_{B_k^K}^0 - x_{B_k^K}^1)$ is zero if $K$ misses one of the colours $1, \ldots, k$.
	Let $\alpha_{\bB, > 0}(X_t)$ be the restriction of $\alpha_{\bB}(X_t)$ where the sum is taken only over cliques $K \in \K_t$ which contain all colours $1, \ldots, k$, but not the colour 0, and let $\K^{\text{rbw}}_{\bB}$ be the collection of all rainbow copies of $K_t$ in $K_n$ under the partition $\bB$.
	For each rainbow $t$-clique $K \in \K^{\text{rbw}}_{\bB}$, let $f^K$ be the only edge in $K$ which is in $B_0$.
	For $K \in \K^{\text{rbw}}_{\bB}$, we set
	\begin{align*}
		C_K = \prod_{e \in E(K) \setminus \{ f^K\}} (x_{e}^0 - x_{e}^1) \qquad \text{and} \qquad
		\alpha_{\bB}^{\text{rbw}}(X_t) = \sum_{K \in \K^{\text{rbw}}_{\bB}} C_K x_{f^K}^0.
	\end{align*}
	It follows from our analysis combined with~\eqref{eq:alpha} that $\alpha_{\bB}(X_t) = \alpha_{\bB}^{\text{rbw}}(X_t) + \alpha_{\bB, > 0}(X_t)$.
	As $\alpha_{\bB, > 0}(X_t)$ is independent of $X^0_{B_0}$, 
	Theorem~\ref{th:decoupling} implies that
	\begin{align}\label{eq:decoupling-rainbow-1}
	\left |\Ex{e^{isX_t}} \right |^{2^k}
        & \le  \mathbb{E}_{Y} \left | \mathbb{E}_{X} \big ( e^{is (\alpha_{\bB}^{\text{rbw}}(X_t) + \alpha_{\bB, > 0}(X_t))} \big )  \right | \nonumber \\
        & \le
        \mathbb{E}_{Y} \left | e^{is\alpha_{\bB, > 0}(X_t)} \mathbb{E}_{X} \big ( e^{is\alpha_{\bB}^{\text{rbw}}(X_t)} \big )  \right | \nonumber \\
        & \le \mathbb{E}_{Y} \left | \mathbb{E}_{X} \big ( e^{is \alpha_{\bB}^{\text{rbw}}(X_t)} \big )  \right |,
	\end{align}
	for all $s \in \mathbb{R}$, where $X = X^0_{B_0}$ and $Y = (\mathds{1}_{e \in G^{i}}: e \in E(K_n) \setminus B_0, i \in \{0,1\})$.

	Observe that the coefficients $C_K$ only depend on $Y$ and the variables $x_{f^K}^0$ only depend on $X$.
	As all rainbow $t$-cliques are edge-disjoint, the random variables $x_{f^K}^0$ are independent, and hence it follows from~\eqref{eq:decoupling-rainbow-1} that
	\begin{align}\label{eq:decoupling-rainbow-2}
		\left |\Ex{e^{isX_t}} \right |^{2^k} \le 
		\mathbb{E}_{Y} \prod_{K \in \K^{\text{rbw}}_{\bB}} \left | \mathbb{E}_{X} e^{is C_K x_{f^K}^0 }  \right |
	\end{align}
	for all $s \in \mathbb{R}$.
    
	Now we use the following bound on the characteristic function of Bernoulli random variable.
	For a proof, see~\cite[Lemma 1]{gilmer2016local}.
	Below, $\|x\|$ denotes the distance from $x$ to the nearest integer.
	
	\begin{lemma}\label{lemma:charfuncbernoulli}
		Let $p \in [0,1]$ and $\theta \in \mathbb{R}$. If $B$ is a $p$-Bernoulli random variable, then
		\[\big |\Ex{e^{i\theta B}} \big | \le 1-8p(1-p)\left \| \dfrac{\theta}{2\pi} \right \|^2 .\]
	\end{lemma}

	\begin{proof}[\textbf{Proof of Theorem~\ref{th:characteristic_function}}]
	By Lemma~\ref{lemma:charfuncbernoulli} combined with~\eqref{eq:decoupling-rainbow-2}, we obtain
	\begin{align}\label{eq:decoupling-rainbow-3}
		\left |\Ex{e^{isX_t}} \right |^{2^k} \le 
		\mathbb{E}_{Y} \left (\prod_{K \in \K^{\text{rbw}}_{\bB}} \left ( 1-8p(1-p)\left \| \dfrac{C_K s}{2\pi } \right \|^2 \right ) \right ).
	\end{align}
	Let $R$ be the random variable which counts the number of cliques in $\K^{\text{rbw}}_{\bB}$ for which $|C_K| = 1$.
	As $1/2 \ge |s/(2\pi)| = \| s/(2\pi) \|$ for all $s \in [-\pi, \pi]$, it follows from~\eqref{eq:decoupling-rainbow-3} that
	\begin{align}\label{eq:decoupling-rainbow}
		\left |\Ex{e^{isX_t}} \right |^{2^k} \le 
		\mathbb{E}_{Y} \left ( \exp \left ( -8p(1-p) \dfrac{s^2R}{(2\pi)^2 } \right ) \right )
	\end{align}
	for all $s \in [-\pi, \pi]$.

	Note that $|C_K| = 1$ if and only if $x^0_e \neq x^1_e$ for all $e \in E(K) \setminus \{ f^K\}$, which occurs with probability $(2p(1-p))^{\binom{t}{2}-1}$.
	As all rainbow $t$-cliques are edge-disjoint, $R$ is distributed as a binomial random variable with parameters $|\K^{\text{rbw}}_{\bB}|$ and $q:= (2p(1-p))^{\binom{t}{2}-1}$.
	Then, denoting $a := 8p(1-p)/(2\pi)^2$, we can rewrite the right-hand side of~\eqref{eq:decoupling-rainbow} as $(1-q + qe^{-as^2})^{|\K^{\text{rbw}}_{\bB}|}$. Therefore, 
	\begin{align*}
		\left |\Ex{e^{isX_t}} \right |^{2^k} & \le (1-q + qe^{-as^2})^{|\K^{\text{rbw}}_{\bB}|} \le \exp \left ( -q(1-e^{-as^2})|\K^{\text{rbw}}_{\bB}| \right ) \le \exp \left ( -\frac{1}{2} q a s^2 |\K^{\text{rbw}}_{\bB}| \right ),
	\end{align*}
	which implies that $|\Ex{e^{isX_t}}| \le \exp ( -\Omega (s^2|\K^{\text{rbw}}_{\bB}|))$.
	The desired bound now follows from the lower bound on the number of rainbow $t$-cliques $|\K^{\text{rbw}}_{\bB}|$ under our partition.
	\end{proof}

\end{document}